\documentclass[a4paper,reqno]{amsart}

\usepackage{amsmath, amsfonts, amssymb, amsxtra, multicol, physics}

\usepackage[utf8]{inputenc}

\usepackage[english]{babel}
\usepackage{hyperref}
\usepackage[mathscr]{eucal}
\usepackage{tikz-cd}
\usepackage{comment}
\usepackage{enumitem}
\usepackage{seqsplit}
\usepackage{dsfont}
\usepackage{bm}

\hypersetup{colorlinks=true,linkcolor=red, anchorcolor=green, citecolor=cyan, urlcolor=red, filecolor=magenta, pdftoolbar=true}

\numberwithin{equation}{section}

\newtheorem{theorem}{Theorem}
\numberwithin{theorem}{section}
\newtheorem{proposition}[theorem]{Proposition}
\newtheorem{lemma}[theorem]{Lemma}

\newtheorem{introtheorem}{Theorem}

\theoremstyle{definition}\newtheorem{definition}[theorem]{Definition}
\theoremstyle{definition}\newtheorem{remark}[theorem]{Remark}
\theoremstyle{definition}\newtheorem{example}[theorem]{Example}
\theoremstyle{definition}\newtheorem*{notation*}{Notation}
\theoremstyle{definition}\newtheorem*{convention*}{Convention}
\theoremstyle{definition}
\theoremstyle{definition}\newtheorem*{acknowledgment*}{Acknowledgments}

\newcommand{\R}{\mathbb{R}}
\newcommand{\C}{\mathbb{C}}

\newcommand{\B}{\mathsf{B}}
\newcommand{\G}{\mathcal{G}}

\newcommand{\cst}{C^*}

\newcommand{\tx}{\mathcal{T}_X}
\newcommand{\ox}{\mathcal{O}_X}
\renewcommand{\op}[1]{\operatorname{#1}}

\DeclareMathOperator{\Aut}{Aut}
\DeclareMathOperator{\Iso}{Iso}
\DeclareMathOperator{\Span}{span}
\DeclareMathOperator{\id}{id}

\begin{document}

\title{KMS states on quantum Cuntz-Krieger algebras}

\begin{abstract}
We study the KMS states on local quantum Cuntz-Krieger algebras associated to  quantum graphs. Using their isomorphism to the Cuntz-Pimsner algebra of the quantum edge correspondence,  we show that the general criteria for KMS states can be translated into statements about the underlying quantum adjacency operator, somewhat analogously to the case of classical Cuntz-Krieger algebras. We study some examples of gauge actions, for which a complete classification of KMS states can be obtained.    
\end{abstract}

\author[M.\,Kumar]{Manish Kumar}
\email{manish.kumar@kuleuven.be}
\address{KU Leuven, Department of Mathematics, Celestijnenlaan 200B, 3001 Leuven, Belgium}

\author[M.\,Wasilewski]{Mateusz Wasilewski}
\email{mwasilewski@impan.pl}
\address{Institute of Mathematics of the Polish Academy of Sciences,
	ul.~\'Sniadeckich 8, 00--656
 Warsaw, Poland}



\maketitle
\section{Introduction}
The study of KMS states has a long history, both in mathematical physics and operator algebra theory. In this article we study these objects for $C^{\ast}$-algebras associated to quantum graphs, namely the \emph{local quantum Cuntz-Krieger algebras} (see \cite{BHINW}).  We choose to work with these rather than the original quantum Cuntz-Krieger algebras (see \cite{BEVW}) because the local versions can be identified with certain Cuntz-Pimsner algebras (see \cite{Pim}), for which one can find a general condition for existence of KMS states in \cite{LN}.

Classical Cuntz-Krieger algebras (\cite{CK}) can also be realized as Cuntz-Pimsner algebras of the edge $\cst$-correspondence  and the general result on gauge actions and their KMS states then reduces to the condition that can be stated purely in terms of the adjacency matrix of the graph (see \cite{EFW}, and also \cite{OP} for the special case of Cuntz algebras). It turns out that in the  case when the adjacency matrix is irreducible there is a unique equilibrium inverse temperature $\beta$ and it is equal to the logarithm of the spectral radius of the  matrix. The aim of our paper is to perform a similar reduction in the case of quantum graphs. Quantum graphs are a non-commutative generalization of classical finite graphs without multiple edges, where we mean a triple  $(\B, A, \psi)$  consisting of a finite-dimensional $\cst$-algebra $\B$, a state $\psi$ on $\B$ and a quantum adjacency operator $A$ (see Subsection \ref{sub:quantum graph}).

We begin with the case of the usual gauge action, where the generators are multiplied by the number $e^{it}$.   In this case we obtain the following result.
\begin{introtheorem}
Let $\mathcal{G}:=(\B, A, \psi)$ be a quantum graph, where $\B\simeq \bigoplus_{a=1}^{d} M_{n_a}$ and let $(\gamma_t)$ be the gauge action on the associated local quantum Cuntz-Krieger algebra $\mathcal{O}_{E_{\mathcal{G}}}$. Then the KMS states on $\mathcal{O}_{E_{\mathcal{G}}}$ are in one-to-one correspondence with positive eigenvectors of a certain integer-valued matrix $D \in M_{d}(\mathbb{N})$ built from $A$. In particular, if the matrix $D$ is irreducible then the KMS state is unique.
\end{introtheorem}

This shows that identifying criteria for KMS states with respect to gauge actions ultimately boils down  to  understanding certain classical graphs with multiple edges.
We prove this theorem by computing explicitly the induced trace from Theorem \ref{Thm:LN}, which is a crucial step in actually classifying the KMS states. Because of the concrete form, for other type of actions we are also able to derive a condition for KMS states, provided that we understand the generator of the action well enough. 

In Subsection \ref{subsec:orth} we take a closer look on the edge correspondences. The quantum edge correspondence, denoted $E_\G$ and  introduced in \cite{BHINW},  is a cyclic sub-$\cst$-correspondence of $\B\otimes_\psi\B$ and generalizes the usual edge correspondence for a classical graph.
\begin{introtheorem}
Let $\mathcal{G}:=(\B, A, \psi)$ be a quantum graph and let $E_{\mathcal{G}}$ be its edge correspondence. Then we can construct an orthonormal basis of $E_{\mathcal{G}}$ given a Kraus decomposition of $A$ consisting of mutually orthogonal operators (see Proposition \ref{Prop:qadjformula} for details).    
\end{introtheorem}
This result allows us to reprove the previous results about KMS states, because having an explicit orthonormal basis is another way to compute the induced trace. We also handle the case of a more general gauge action.

In the last part of the paper we provide an example of a gauge action on $\mathcal{O}_{E_{\mathcal{G}}}$, whose restriction to $\B$ is non-trivial, using the modular group. In this case we no longer work with induced traces but with more general functionals. In this case once again the KMS states are governed by a certain matrix $D \in M_d(\mathbb{R})$, but this time it is more complicated, e.g. it depends on the inverse temperature $\beta$. In the case of the complete quantum quantum graph the situation simplifies and a more satisfactory answer can be obtained.

\section{Preliminaries}
\subsection{Cuntz-Pimsner algebras}

Let $\B$ be a $\cst$-algebra.
A {\em $\cst$-correspondence} over  $\B$ is a right Hilbert $\B$-module $X$ together with a $\B$-valued inner product $\langle\cdot,\cdot\rangle$ and a $\ast$-homomorphism $\phi_X:\B\to\op{B}(X)$, where $\op{B}(X)$ is the algebra of all adjointable operators on $X$. We simply write $b\xi$ for $\phi_X(b)\xi$, $b\in\B,\xi\in X$.  We follow the convention of linearity in the second coordinate of the inner product. We now briefly describe below the construction of Toeplitz-Pimsner and Cuntz-Pimsner algebras  associated to a $C^*$-correspondence over $\B$ (see \cite{Pim} for more details on the topic).

The {\em Toeplitz-Pimsner algebra} $\mathcal{T}_X$ is the universal $\cst$-algebra generated by elements $\pi(b)$ and $T_\xi$ with $b\in\B$ and $\xi\in X$ such that $\pi:\B\to\tx$ is a $\ast$-homomorphism, $T_{a\xi b}=\pi(a)T_\xi\pi(b)$ and $T_{\xi}^*T_\eta=\pi(\langle\xi,\eta\rangle)$ for $a,b\in\B$ and $\xi,\eta\in X$. A concrete way of constructing such algebras is as follows:
Let $\mathcal{F}(X)=\B\oplus\bigoplus_{n\geq 1}X^{\otimes_\B n}$ where $X^{\otimes_\B n}$ denotes the $n$-times (internal) tensor product of the $\cst$-correspondence $X$. Consider the $\ast$-homomorphism $\pi:\B\to\op{B}({\mathcal{F}(X)})$ and the left creation operators $T_\xi\in \op{B}({\mathcal{F}(X)})$, $\xi\in X$  given by $\pi(a)\eta=a\eta$ and $T_\xi(\eta)=\xi\otimes\eta$ for $a\in\B$, $\xi\in \mathcal{F}(X)$. Then $\tx$ is the $\cst$-subalgebra of $\op{B}(\mathcal{F}(X))$ generated by $\pi(a), T_\xi$.

The {\em Cuntz-Pimsner algebra} $\mathcal{O}_X$ is the quotient of $\tx$ by the ideal generated by elements of the form $\pi(b)-j_X(\phi_X(b))$ for $b\in I_X$ where \[I_X=\{a\in \B; \phi_X(a)\in \mathcal{K}(X) \mbox{ and } ab=0\;\forall b\in \ker\phi_X\}.\] Here $\mathcal{K}(X)$ is the space of compact operators on $X$ generated by $|\xi\rangle\langle\eta|\in \op{B}(X)$, $\xi,\eta\in X$, where $|\xi\rangle\langle\eta|(\zeta)=\xi\langle\eta,\zeta\rangle$ for $\zeta\in X$, and $j_X:\mathcal{K}(X)\to\tx$ is the homomorphism given by $j_X(|\xi\rangle\langle\eta|)=T_\xi T_\eta^*$ for $\xi,\eta\in X$.

Now let $\sigma:\R\to   \Aut (\B)$ be a one-parameter group of automorphisms of $\B$, and let $U:\R\to \Iso(X)$ be a one-parameter group of isometries on $X$ such that
\[U_t (a\xi)=\sigma_t(a)U_t\xi\;\;\;\mbox{ and }\;\;\langle U_t\xi, U_t\eta\rangle=\sigma_t(\langle\xi,\eta\rangle).\]
Moreover, both are assumed to be strongly continuous i.e. $t\mapsto\sigma_t(a)$ and $t\mapsto U_t\xi$ are continuous for all $a\in \B$ and $\xi\in X$.
By the universal property of the Toeplitz-Pimsner algebra, there exists a (unique) automorphism $\gamma_t$ on $\tx$ such that $\gamma_t(\pi(a))=\pi(\sigma_t(a))$ and $\gamma_t(T_\xi)=T_{U_t\xi}$
 for all $a\in \B$ and $\xi\in X$. It is immediate that $t\mapsto\gamma_t$ is a strongly continuous one-parameter automorphism group of $\tx$.

 Further note that the expression $U_t(a\xi)=\sigma_t(a)U_t\xi$ for all $\xi\in X$ precisely means that  $\phi_X(\sigma_t(a))=U_t\phi_X(a)U_t^*$. This implies that $a\in \ker\phi_X$ iff $\sigma_t(a)\in \ker\phi_X$, and $\phi_X(a)\in \mathcal{K}(X)$ iff $\phi_X(\sigma_t(a))\in \mathcal{K}(X)$ such that $\gamma_t(j_X(\phi_X(a)))=j_X(U_t\phi_X(a)U_t^*)=j_X(\phi_X(\sigma_t(a)))$. This shows that $a\in I_X$ iff $\sigma_t(a)\in I_X$ and $\gamma_t(I_X)=I_X$. Hence $\gamma_t$ further induces a strongly continuous one-parameter group of automorphisms on the quotient Cuntz-Pimsner algebra $\ox$.

\subsection{Quantum graphs}\label{sub:quantum graph}

Let $\B$ be a finite dimensional $\cst$-algebra equipped with a faithful state $\psi:\B\to\C$. We denote by $L^2(\B,\psi)$ or simply by $L^2(\B)$  the corresponding GNS Hilbert space, which as a  set is nothing but $\B$ itself. Let $m:L^2(\B)\otimes L^2(\B)\to L^2(\B)$ be the multiplication map, and $m^*:L^2(\B)\to L^2(\B)\otimes L^2(\B)$ be its adjoint. For $\delta>0$, the state $\psi:\B\to\C$ is called a {\em $\delta$-form} if $mm^*=\delta^2 \id$.

\begin{definition}
    A {\em quantum graph} $\G$ is a triple $(\B, A,\psi)$ where $\B$ is a finite-dimensional $\cst$-algebra, $\psi:\B\to\C$ is a $\delta$-form, and $A:\B\to\B$ is a {\em quantum adjacency operator} i.e. it satisfies $m(A\otimes A)m^*=\delta^2 A$ and is $\ast$-preserving (or, equivalently, completely positive).
\end{definition}

Let $\G=(\B, A, \psi)$ be a quantum graph, where $\psi:\B\to\C$ is  a $\delta$-form. Write 
\[\B=\bigoplus_{a=1}^dM_{n_a}\] 
and 
\[\psi=\bigoplus_{a=1}^d\Tr(\cdot\rho_a)\] for some natural numbers $d, n_a$ and positive invertible operators $\rho_a\in M_{n_a}$ such that $\sum_{a=1}^d\Tr(\rho_a)=1$. Here $\Tr$ denotes the (non-normalized) traces on $M_{n_a}$. We assume that each $\rho_a$ is a diagonal operator with diagonal entries equal to $\psi(e_{11}^a), \ldots, \psi(e_{n_an_a}^a)$. Here $e_{ij}^a$ denotes the matrix units of $M_{n_a}$. That $\psi$ is a $\delta$-form then means that $\Tr(\rho_a^{-1})=\delta^2=\sum_{k=1}^{n_a}\psi(e_{kk}^a)^{-1}$ for all $1\leq a\leq d.$ Set
\[f_{ij}^a=\frac{e_{ij}^a}{\psi(e_{ii}^a)^{1/2}\psi(e_{jj}^a)^{1/2}}.\]
Then we have
\begin{align*}
    &m^*(f_{ij}^a)=\sum_{k=1}^{n_a}f_{ik}^a\otimes f_{kj}^a\\
    & f_{ij}^af_{rs}^b=\delta_{ab}\delta_{jr}\frac{f_{is}^a}{\psi(e_{jj}^a)}
\end{align*}
for $1\leq a,b\leq d$ and $1\leq i,j\leq n_a, 1\leq r,s\leq n_b$. Here $\delta_{ab}$ denotes the Kronecker delta.

\subsection{Quantum Edge Correspondence} Let $\G=(\B, A, \psi)$ be a quantum graph such that $\psi$ is a $\delta$-form. Consider the $\cst$-correspondence $\B\otimes_\psi\B$ which as a vector space is equal to $\B\otimes\B$ as $\psi$ is faithful; the $\B$-valued inner product is defined by $\langle a\otimes b, c\otimes d\rangle:= \psi(a^{\ast}c)b^{\ast} d$. We now consider a sub $\cst$-correspondence of $\B\otimes_\psi\B$ as follows.

Consider the {\em quantum edge indicator} $\varepsilon_\G$ given by 
\[\varepsilon_\G=\frac{1}{\delta^2}(\id \otimes A)m^*(1)\in \B\otimes_\psi\B.\]
The {\em quantum edge correspondence} $E_\G$ is the cyclic bimodule generated by $\varepsilon_\G$ i.e.
\[E_\G=\B\varepsilon_\G\B=\Span\{a\varepsilon_\G b; a,b\in \B\}.\]
Recall from \cite[Theorem 2.9]{BHINW} that $E_\G$ is faithful (i.e. the left action is a faithful homomorphism) iff $\ker A$ does not contain a central summand of $\B$, and $E_\G$ is full (i.e.  $\Span\langle E_\G,E_\G\rangle=\B$) iff range of $A$ is not orthogonal to a central summand of $\B$.

\subsection{Quantum Cuntz-Krieger algebras}
We now recall the definition of quantum Cuntz-Krieger algebras as defined in \cite{BEVW} (note that the third condition does not appear there, but it is a natural unitality condition that we adopt, see also \cite[Definition 3.1]{BHINW}).
\begin{definition}
    Let $\G=(\B, A, \psi)$ be a quantum graph. The {\em quantum Cuntz-Krieger algebra} associated to $\G$ is defined to be the universal unital $\cst$-algebra $\mathcal{O}(\G)$ generated by the elements $S(b), b\in \B$ where $S:\B\to\mathcal{O}(\G)$ is a linear map satisfying the following:
    \begin{enumerate}
        \item $\mu(\mu\otimes \id)(S\otimes S^*\otimes S)(m^*\otimes \id)m^*=S$
        \item $\mu(S^*\otimes S)m^*=\mu(S\otimes S^*)m^*A$
        \item $\mu(S\otimes S^*)m^*(1)=\frac{1}{\delta^2}\id$.
    \end{enumerate}
    Here $S^*$ is the map $S^*:\B\to\mathcal{O}(\G)$ given by $S^*(b)=S(b^*)^*$ and $\mu:\mathcal{O}(\G)\otimes\mathcal{O}(\G)\to\mathcal{O}(\G)$ is the multiplication operator.
\end{definition}

We will impose formally stronger relations, called the local Cuntz-Krieger relations in \cite[Definition 3.4]{BHINW}, namely the first two equalities are replaced by $\mu(\mu\otimes \id)(S\otimes S^*\otimes S)(m^*\otimes \id)=\frac{1}{\delta^2}Sm$ and $\mu(S^*\otimes S)=\frac{1}{\delta^2}\mu(S\otimes S^*)m^* Am$. They ensure that we obtain a $C^{\ast}$-algebra isomorphic to the Cuntz-Pimsner algebra associated to the quantum edge correspondence under mild conditions (that the quantum edge correspondence is faithful). In cases that both versions of quantum Cuntz-Krieger algebras are understood, they are isomorphic. We recall the following results:

\begin{theorem}[{\cite[Theorem 3.6]{BHINW}}]\label{thm:local CKA and CPA}
    Let $\G=(\B, A,\psi)$ be a quantum graph such that $\psi$ is a $\delta$-form. Let $(\pi_\G, t_\G)$ denote the universal covariant representation of the quantum edge correspondence $E_\G$. Assume that $E_\G$ is faithful.  Then $\mathcal{O}_{E_\G}$ is the local quantum Cuntz-Krieger algebra of $\G$ with the associated local quantum Cuntz-Krieger family  $S:\B\to \mathcal{O}_{E_\G}$ given by $S(b)=\frac{1}{\delta}t_{\G}(b\varepsilon_\G)$ for $b\in \B$.
\end{theorem}

In this work we deal exclusively with the KMS states on Cuntz-Pimsner algebras associated to the quantum edge correspondence and then we use the above theorem to induce KMS states on the associated (local) quantum Cuntz-Krieger algebras. We recall the following definition.

\begin{definition}
    Let $\mathsf{D}$ be a $\cst$-algebra, and let $\sigma=(\sigma_t)_{t\in\R}$ be a one-parameter group of automorphisms on $\mathsf{D}$. Let $\beta\in\R$. A state $\varphi:\mathsf{D}\to\C$ is said to be a $(\sigma,\beta)$-{\em KMS state} (or a KMS state at the inverse temperature point $\beta$) if $\varphi$ is $\sigma_t$ invariant and satisfies :
    \[\varphi(ab)=\varphi(b\sigma_{i\beta}(a))\]
    for all $b\in \mathsf{D}$ and entire vectors $a\in \mathsf{D}$. 
\end{definition}

\section{KMS states (tracial setting)}

We want to determine the KMS states of quantum Cuntz-Pimsner algebras with respect to some natural dynamical systems.
Let $\G=(\B, A, \psi)$ be a quantum graph such that $\psi$ is a $\delta$-form. Write
\[\B=\bigoplus_{a=1}^dM_{n_a}\;\;\mbox{and }\psi=\bigoplus_{a=1}^d\Tr(\rho_a\cdot)\]
as above. Let $\sigma=(\sigma_t)_{t\in \R}$ be a one-parameter group of automorphisms on $\B$ and let $(U_t)_{t\in \R}$ be a one-parameter group of invertible isometries on $E_\G$ such that $U_t(b\xi)=\sigma_t(b)U_t\xi$ and $\langle U_t\xi, U_t\eta\rangle=\sigma_t(\langle\xi,\eta\rangle)$. Let $\gamma=(\gamma_t)_{t\in\R}$ be the induced one-parameter groups of automorphisms on $\mathcal{O}_{E_\G}$.

We first assume that $\sigma$ is trivial so that the $\beta$-KMS states on $\B$ are tracial states. In this case we will use the following theorem (slightly simplified because we are working in the finite dimensional setting, we also did not include the positive energy condition in the statement, which will be satisfied in all of the examples we consider).

\begin{theorem}[{\cite[Theorems 2.1, 2.5]{LN}}]\label{Thm:LN}
Let $X$ be a $\B$-correspondence and let $(U_t)_{t\in\R}$ be a one-parameter group of bimodular isometries on $X$. We will write $U_{t} = \exp(itD)$, where $D$ is the generator. Let $\gamma$ be the induced action on $\mathcal{O}_X$, which is trivial on $\B$. Then $\phi: \mathcal{O}_{X} \to \mathbb{C}$ is a $(\gamma,\beta)$ KMS state iff its restriction $\tau:=\phi_{|\B}$ is a trace such that $\op{Tr}_{\tau}(b \exp(-\beta D)) = \tau(b)$, where $\op{Tr}_{\tau}$ is the trace on $\op{B}(X)$ given by $\op{Tr}(|\xi\rangle\langle \xi|):= \tau(\langle \xi, \xi\rangle)$ (see \cite[Theorem 1.1.]{LN}). Moreover, the state $\phi$ is determined as:
\begin{align*}
    \phi(T_{\xi_1}\cdots T_{\xi_n}T_{\eta_m}\cdots T_{\eta_1}^*)=\delta_{n,m}\tau(\langle\eta_1\otimes\cdots\otimes\eta_n, e^{-\beta D}\xi_1\otimes\cdots \otimes e^{-\beta D}\xi_n\rangle)
\end{align*}
for $\xi_i,\eta_j\in X$.
\end{theorem}

\begin{remark}
1.
 We are first interested in the $\beta$-KMS states on $\mathcal{O}_{E_\G}$ with respect to the {\em gauge action
 } i.e. when $U_t=e^{it}$ for all $t\in \R$. On the other hand, there is a natural  gauge action  $(\Tilde{\gamma}_t)_{t\in \R}$  on the (local) quantum Cuntz-Krieger algebra given by $\tilde{\gamma}_t(S(b))=e^{it}S(b)$. If $E_\G$ is faithful, by the universal property, it is clear that the isomorphism between $\mathcal{O}_{E_\G}$ and the local quantum Cuntz-Krieger algebra (as described in Theorem \ref{thm:local CKA and CPA} above) is an equivarient map under the gauge actions. Thus there is a one-to-one correspondence between the KMS states  on the two algebras with respect to the gauge actions. Moreover, such KMS state $\phi$ on the local Cuntz-Krieger algebra will be determined by 
 \begin{align*}
    \phi(S(a_1)\cdots S(a_n)S(b_m)^*\cdots S(b_1)^*)&= \frac{1}{\delta^{2n}}\delta_{n,m} e^{-n\beta}\tau(\langle b_1\varepsilon_\G\otimes\cdots \otimes b_n\varepsilon_\G, a_1\varepsilon_\G\otimes\cdots \otimes a_n\varepsilon_\G\rangle)\\
   &=\frac{1}{\delta^{2n}}\delta_{n,m} e^{-n\beta}\tau(A(b_n^*A(b_{n-1}^*\cdots A(b_2^*A(b_1^*a_1)a_2)\cdots a_{n-1})a_n)).
 \end{align*}
 for $a_i, b_j\in \B$.
 
 2. Since $\mathcal{O}_{E_\G}$ is a quotient of the quantum Cuntz-Krieger algebra (see \cite[Corollary 3.7]{BHINW}) and the quotient map is clearly gauge action equivarient, we will also obtain a KMS state on the quantum Cuntz-Krieger algebra.
\end{remark}


Let us now concretely obtain a criterion for the existence of $\beta$-KMS states on $\mathcal{O}_{E_\G}$ with respect to the gauge action.
Let $\tau:=\bigoplus_{a=1}^d\lambda_a\Tr(\cdot)$ be a tracial state on $\B$, where $\lambda_a\geq0$ and $\sum_{a=1}^d\lambda_a n_a=1$. Note that for any $X=\oplus_{a=1}^{d}X_a\in\B$ with $X_a\in M_{n_a}$, we have
\begin{align*}
    \tau(X)&=\sum_{a=1}^d\lambda_a\Tr(X_a)=\sum_{a=1}^d\lambda_a\psi(\rho_{a}^{-1}X_a)
=\sum_{a=1}^d\lambda_a\langle\rho_{a}^{-1},X_a\rangle_\psi
\\&=\sum_{a=1}^d\lambda_a\langle\rho_a^{-1}, X\rangle_\psi=\langle\oplus_{a=1}^d\lambda_a\rho_a^{-1}, X\rangle_\psi.
\end{align*}
The left action of $f_{ij}^a$ on $E_\G$ is given by
\[f_{ij}^a\eta=\sum_{k=1}^{n_a}|f_{ik}^a\varepsilon_\G\rangle\langle f_{jk}^a\xi_G|(\eta)=\sum_{k=1}^{n_a}f_{ik}^a\varepsilon_\G\langle f_{jk}^a\varepsilon_\G,\eta\rangle,\;\;\;\;\eta\in E_\G.\]
(see \cite[Theorem 2.12]{BHINW}). Consider the trace $\Tr_\tau$ defined on $\op{B}(E_\G)$ as in \cite[Theorem 1.1]{LN}. By using the same theorem from \cite{LN} and \cite[Theorem 2.5]{BHINW}, we calculate
\begin{align*}
    \Tr_\tau(f_{ij}^a)&=\sum_{k=1}^{n_a}\Tr_\tau(|f_{ik}^a\varepsilon_\G\rangle\langle f_{jk}^a\varepsilon_\G|)=\sum_{k=1}^{n_a}\tau(\langle f_{jk}^a\varepsilon_\G, f_{ik}^a\varepsilon_\G\rangle)\\
    &=\sum_{k=1}^{n_a}\frac{1}{\delta^2}\tau(A((f_{jk}^a)^*f_{ik}^a))=\frac{1}{\delta^2}\sum_{k=1}^{n_a}\frac{1}{\psi(e_{ii}^a)}\tau(A(\delta_{ij}f_{kk}^a))\\
    &=\frac{\delta_{ij}}{\psi(e_{ii}^a)}\frac{1}{\delta^2}\tau(A(\rho_a^{-1}))
\end{align*}
where we  use the expression $\rho_a^{-1}=\sum_{k=1}^{n_a}\psi(e_{kk}^a)^{-1}e_{kk}^a=\sum_{k=1}^{n_a}f_{kk}^a$ as each $\rho_a$ is a diagonal matrix. Thus we get
\[\Tr_\tau(e_{ij}^a)=\delta_{ij}\frac{1}{\delta^2}\tau(A(\rho_a^{-1})).\]
On the other hand, we have
\begin{align*}
    \tau(e_{ij}^a)=\lambda_a\Tr(e_{ij}^a)=\lambda_a\delta_{ij}.
\end{align*}
  Since $\beta$ is a KMS temperature point for a state on  $\mathcal{O}_{E_\G}$ if and only if $\Tr_\tau(be^{-\beta})=\tau(b)$ for all $b\in \B$, the calculations above show that the last condition is equivalent to $\frac{1}{\delta^2}\tau(A(\rho_a^{-1}))=e^\beta\lambda_a,\forall 1\leq a\leq d$ that is,
\begin{align*}
  \frac{1}{\delta^2}\sum_{b=1}^d\lambda_b\langle\rho_b^{-1},A(\rho_a^{-1})\rangle_\psi=e^\beta\lambda_a,\;\;\;\forall 1\leq a\leq d.
\end{align*}
If we consider  the  matrix $D=[D_{ab}]_{1\leq a,b\leq d}$ in $M_d$ given by  $D_{ab}=\langle\rho_b^{-1}, A(\rho_{a}^{-1})\rangle_\psi$, $\lambda\in \C^d$ is the vector given by $\lambda=(\lambda_1,\ldots,\lambda_d)$, and $p_a$ is the canonical projections in $\C^d$, then above condition is equivalent to the following: 
\begin{align*}
     \frac{1}{\delta^2}\langle p_a, D\lambda \rangle_{\mathbb{C}^d}=e^\beta\langle p_a,\lambda\rangle_{\C^d},\;\;\;\forall1\leq a\leq d
\end{align*}
i.e $D\lambda=\delta^2e^\beta\lambda$.
\begin{remark}
 1. Since $A$ is a positive map, $A(\rho_b^{-1})\geq 0$ in $\B$.  As $\langle\rho_b^{-1},A(\rho_a^{-1})\rangle_\psi$ is nothing but the trace of $b$th component of $A(\rho_a^{-1})$, it follows that $\langle\rho_b^{-1},A(\rho_a^{-1})\rangle_\psi\geq 0$ i.e. $D$ is a matrix with non-negative entries.

2. If we demand $\tau=\oplus_{a=1}^d\lambda_a\Tr(\cdot)$ to be faithful, then $\lambda_a>0$ for all $a.$ In this case, the equality  $D\lambda=\delta^2e^\beta\lambda
$ can only occur if $\delta^2 e^\beta=r(D)$, the spectral radius of $D$ (see \cite[Corollary 8.1.30]{HJ}).  In particular, the only possible value of $\beta$ is $\log(\frac{1}{\delta^2}r(D))$. 

 3. If the matrix $D$ is irreducible, then the only possible eigenvectors with non-negative entries correspond to the eigenvalue $r(D)$; in such cases the eigenspace corresponding to $r(D)$  is one-dimensional. Hence, we will have only one choice of $\beta$ and $\lambda_a's$ and hence the $\beta$-KMS state is unique.
\end{remark}

We summarize the above discussion in the following theorem.
\begin{theorem}\label{Thm:KMS}
    Let $\B=\bigoplus_{a=1}^d M_{n_a}$, and let $\G=(\B, A,\psi)$ be a quantum graph. Let $\tau=\bigoplus_{a=1}^d\lambda_a\Tr(\cdot)$ for $\lambda_a\geq 0$ with $\sum_{a=1}^d\lambda_a n_a=1$. Then $\varphi$ is a $\beta$-KMS state on $\mathcal{O}_{E_\G}$ with respect to the gauge action such that $\varphi_{|_\B}=\tau$ iff $\lambda=(\lambda_1,\ldots,\lambda_d)\in\C^d$ is an eigenvector of the matrix $D:=[\langle\rho_a^{-1}, A(\rho_{b}^{-1})\rangle_\psi]$ with the eigenvalue $\delta^2e^\beta$. Moreover, the KMS state is unique if the matrix $D$ is irreducible.
\end{theorem}

Let us now calculate $\beta$ in some special cases for the KMS states with respect to the gauge action.
\\
\\
1. \textbf{(The classical case):} $\B=\C^d$, $\psi(x)=\frac{x_i}{d}$ for $x=(x_1,\ldots,x_d)\in \C^d$ so that $\psi$ is a $\delta$-form for $\delta^2=d$, and $A:\C^d\to\C^d$ is an adjacency matrix. Let $\langle\cdot,\cdot\rangle_{\C^d}$ denote the usual inner product on $\C^d$ so that $\langle\cdot,\cdot\rangle_{\C^d}=d\langle\cdot,\cdot\rangle_\psi$. For $1\leq a\leq d$, we have $\rho_a=\frac{p_a}{d}$ so that $\rho_a^{-1}=dp_a$ where $\{p_a\}_{1\leq a\leq d}$ are the canonical projections in $\C^d$. Hence
\[\langle\rho_b^{-1},A\rho_a^{-1}\rangle_\psi=\frac{1}{d}\langle\rho_b^{-1},A(\rho_a^{-1})\rangle_{\C^d}=d\langle p_b,A(p_a)\rangle_{\C^d}=dA_{ba}\]
where $A=[A_{ab}]$ is written in the matrix form with the canonical basis of $\C^d$. Thus $D=dA^t$ where $A^t$ denotes the transpose of $A$. 
If $A$ is irreducible (i.e. the graph is connected) then the only possible value of $\beta$ is
\[\beta=\log(\frac{1}{\delta^2}r(D))=\log \frac{1}{d}dr(A^t)=\log r(A)\]
and the KMS state is unique. Moreover, if the graph has no source (recall that a source is a vertex with no edges into it) then $E_\G$ is faithful and the KMS state $\phi$ on the corresponding  Cuntz-Krieger  algebra is given by
\begin{align*}
    \phi(S_{i_1}\cdots S_{i_n}S_{j_m}^*\cdots S_{j_1}^*)&=\delta_{m,n}\frac{1}{d^2}e^{-n\beta}\tau(\langle p_{j_1}\varepsilon_\G\otimes\cdots p_{j_n}\varepsilon_\G, p_{i_1}\varepsilon_\G\otimes\cdots \otimes p_{j_n}\varepsilon_\G\rangle)\\
&=\frac{1}{d^2(r(A))^n}\delta_{m,n}\delta_{i_1,j_1}\cdots\delta_{i_n,j_n}A_{j_2j_1} A_{j_3j_2}\cdots A_{j_nj_{n-1}}\lambda_{j_m} 
\end{align*}
where $(\lambda_1,\cdots,\lambda_d)$ is an eigenvector of $A$ corresponding to the eigenvalue $r(A)$, and $S_i=S(p_i)$ for $1\leq i\leq d$.
\\
\\
2. \textbf{(Complete quantum graph): }$\B=\oplus_{a=1}^d{M_{n_a}}$, $\psi:\B\to\C$ a $\delta$-form and $A(\cdot)=\delta^2\psi(\cdot)1_\B$. Then, for $1\leq a,b\leq d$, we have
\[\langle\rho_b^{-1},A(\rho_a^{-1})\rangle_\psi=\delta^2\langle\rho_b^{-1}, \psi(\rho_a^{-1})1_\B\rangle_\psi=\delta^2\psi(\rho_a^{-1})\psi(\rho_{b}^{-1})=\delta^2 n_an_b\]
so that $D=\delta^2[n_an_b]$. Hence $D$ is an irreducible matrix with $r(D)=\delta^2(\sum_{a=1}^dn_a^2)$. So we get
\[\beta=\log(\sum_{a=1}^dn_a^2)\]
and the KMS state is unique. Clearly $E_\G$ is faithful and the KMS state $\phi$ on the (local) quantum Cuntz-Krieger algebra is determined by
\begin{align*}
    \phi(S(a_1)\cdots S(a_n)S(b_m)^*\cdots S(b_1)^*)=\frac{\delta_{m,n}}{\delta^{2n}}e^{-n\beta}\psi(b_1^*a_1)\cdots\psi(b_n^*a_n).
\end{align*}
for $a_i, b_j\in \B$.
\\
\\
3. (\textbf{Rank-one quantum graph}): $\B=\oplus_{a=1}^d M_{n_a}$, $\psi$  a $\delta$-form and $A(x):=TxT^*$ for some $T\in \B$ such that $\Tr(\rho_a^{-1}T^*T)=\delta^2$ for $1\leq a\leq d$. We have
\[\langle \rho_b^{-1}, A(\rho_a^{-1})\rangle_\psi=\delta_{ab}\psi(\rho_a^{-1}T\rho_a^{-1}T^*)=\delta_{ab}\Tr(\rho_a^{-1}T^*T)=\delta_{ab}\delta^2\]
so that $D=\delta^2 id$. Hence, any vector is an eigenvector of $D$ with the eigenvalue $\delta^2$. So we get $\beta=0$ and each trace on $\B$ gives a tracial KMS state on $\mathcal{O}_{E_{\G}}$. This is not surprising, since by \cite[Proposition 4.10]{BHINW} the local quantum Cuntz-Krieger algebra of a rank-one quantum graph is the same as for the trivial quantum graph. Again $E_\G$ is faithful and  the KMS state $\phi$ is determined as:
\[ \phi(S(a_1)\cdots S(a_n)S(b_m)^*\cdots S(b_1)^*)=\frac{\delta_{m,n}}{\delta^{2n}}\tau(Tb_n^*T\cdots Tb_2Tb_1^*a_1T^*a_2T^*\cdots T^*a_nT^*)\]
for $a_i,b_j\in \B$. 
\\
\\
4. (\textbf{Single matrix block}): $\B=M_n, \tau=\frac{\Tr(\cdot)}{n}$, $\psi=\Tr(\rho\cdot)$ so that $\delta^2=\Tr(\rho^{-1})$. In this case, $D=\langle\rho^{-1}, A(\rho^{-1})\rangle_\psi=\Tr(A(\rho^{-1}))$. The KMS state is always unique with inverse temperature point $\beta=\log \frac{1}{\delta^2}\Tr(A(\rho^{-1}))$.

\subsection{A classical graph with multiple edges}\label{subsec:int} Our next aim is to show that the entries of the matrix $\frac{1}{\delta^2}D$ are always non-negative integers.
For $1\leq b\leq d$ and $1\leq i,j\leq n_b$, write 
\[Af_{ij}^b=\oplus_{a=1}^d X_{ij}^{ab},\;\;\mbox{ for }X_{ij}^{ab}\in M_{n_a}.\]
Set $X^{ab}=[X_{ij}^{ab}]_{1\leq i,j\leq n_b}\in M_{n_b}(M_{n_a})\subseteq M_{n_b}(\B)$.

\begin{lemma}
    The operator $\frac{1}{\delta^2}X^{ab}$ is a projection in $M_{n_b}(\B)$.
\end{lemma}
\begin{proof}
Since $A$ is $\ast$-preserving, we have
\[Af_{ij}^{b}=A((f_{ji}^{b})^*)=(Af_{ji}^{b})^*,\]
that is
\[X_{ij}^{ab}=(X_{ji}^{ab})^*,\;\;\;1\leq i,j\leq n_b\]
so that $X^{ab}=(X^{ab})^*$. We will now show that $\delta^2 X^{ab}=(X^{ab})^2$. We first note that
\begin{align*}
    \delta^2 A(f_{{ij}}^{b})=m(A\otimes A)m^*(f_{ij}^{b})=m(A\otimes A)\left(\sum_{k=1}^{n_b}f_{ik}^b\otimes f_{kj}^b\right)=\sum_{k=1}^{n_b}A(f_{ik}^b)A(f_{kj}^b)
\end{align*}
so we get
\begin{align*}
    &\delta^2\bigoplus_{a=1}^dX_{ij}^{ab}=\bigoplus_{a=1}^d\sum_{k=1}^{n_b}X_{ik}^{ab}X_{kj}^{ab}    
    \end{align*}
    that is
     \begin{align*}
    \delta^2 X_{ij}^{ab}=\sum_{k=1}^{n_b}X_{ik}^{ab}X_{kj}^{ab}.
    \end{align*}
   This shows that  $\delta^2 X^{ab}= (X_{ab})^2$
    so that $\frac{1}{\delta^2}X^{ab}=(\frac{1}{\delta^2}X^{ab})^2$; hence $\frac{1}{\delta^2}X^{ab}$ is a projection.
    \end{proof}

Since $\frac{1}{\delta^2}X^{ab}$ is a projection in $M_{n_b}(\B)$, it follows that its (non-normalized) trace is a natural number. We thus have the following:
\begin{align*}
   \frac{1}{\delta^2}\langle\rho_a^{-1},A(\rho_b^{-1})\rangle_\psi&= \frac{1}{\delta^2}\sum_{k=1}^{n_b}\langle\rho_a^{-1}, A(f_{kk}^b)\rangle_\psi\\
   &= \frac{1}{\delta^2}\sum_{k=1}^{n_b}\langle\rho_a^{-1}, \bigoplus_{c=1}^dX_{kk}^{cb}\rangle_\psi\\
   &= \frac{1}{\delta^2}\sum_{k=1}^{n_b}\langle\rho_a^{-1}, X_{kk}^{ab}\rangle_\psi\\
   &= \frac{1}{\delta^2}\sum_{k=1}^{n_b}\Tr_{M_{n_a}}(X_{kk}^{ab}).
\end{align*}
The last quantity is nothing but the sum of the trace of the diagonal entries of the matrix $ \frac{1}{\delta^2}X^{ab}$; hence it must be a positive integer, as was claimed. 

The above calculations  mean that the task of finding the KMS states reduces to studying classical graphs with multiple edges. In the next subsection we will introduce an orthonormal basis of the edge correspondence, which will help us in reproving the results about KMS states and will provide an interpretation to the fact that the entries of $\frac{1}{\delta^2} D$ are integers.

\subsection{An orthonormal basis of the edge correspondence}\label{subsec:orth}

Let us start with a single matrix block, i.e. $\B \simeq M_n$. We have a state defined by $\psi(x):= \op{Tr}(\rho x)$ and we assume that the density matrix $\rho$ is diagonal. Recall that $m^{\ast}(e_{ij}) = \sum_{k=1}^n e_{ik}\rho^{-1} \otimes e_{kj}$, so a quantum adjacency matrix satisfies
\[
\sum_{k}A(e_{ik} \rho^{-1}) A(e_{kj}) = \delta^2 A(e_{ij}),
\]
where $\delta^2= \op{Tr}(\rho^{-1})$. If $A$ has a Kraus decomposition $Ax = \sum_{r} V_{r} x V_{r}^{\ast}$ for some $V_r\in M_n$ then the condition becomes
\[
\sum_{r,q,k} V_{r} e_{ik} \rho^{-1} V_{r}^{\ast} V_{q} e_{kj} V_{q}^{\ast} = \delta^2 \sum_{r} V_{r} e_{ij} V_{r}^{\ast}.
\]
If we note that $\sum_{k} e_{ik} \rho^{-1} V_{r}^{\ast} V_{q} e_{kj} = \op{Tr}(\rho^{-1} V_{r}^{\ast} V_{q}) e_{ij}$, then we see that the condition $\op{Tr}(\rho^{-1} V_{r}^{\ast} V_{q}) = \delta^2 \delta_{rq}$ is sufficient to satisfy the above equality. We will now explain that one can always choose Kraus operators to satisfy this orthogonality; from now on we will always choose them to do so.
\begin{proposition}[{\cite[Proposition 3.30]{Was}}]\label{Prop:qadjformula}
Let $V \in M_n$ be a subspace and let $(X_1, \dots, X_d)$ be an orthonormal basis of $V$ with respect to the KMS inner product induced by $\psi^{-1}$, i.e. we have $\op{Tr}(X_{i}^{\ast} \rho^{-\frac{1}{2}} X_{j} \rho^{-\frac{1}{2}}) = \delta^2 \delta_{ij}$. Then the corresponding quantum adjacency matrix is given by
\[
A(x):= \sum_{i=1}^{d} \rho^{-\frac{1}{4}} X_i \rho^{\frac{1}{4}} x \rho^{\frac{1}{4}} X_{i}^{\ast}  \rho^{-\frac{1}{4}}.
\]
The Kraus operators $V_{i}:= \rho^{-\frac{1}{4}} X_i \rho^{\frac{1}{4}} $ satisfy $\op{Tr}(\rho^{-1} V_{i}^{\ast} V_{j}) = \delta^2 \delta_{ij}$.

An analogous statement holds if $M_n$ is replaced by a multi-matrix algebra.
\end{proposition}
\begin{proof}
The formula for the quantum adjacency matrix is taken from \cite[Proposition 3.30]{Was}, the only difference being that $\delta=1$ in that work. We just have to check that the Kraus operators satisfy the appropriate orthogonality condition. We have $X_i = \rho^{\frac{1}{4}} V_i \rho^{-\frac{1}{4}}$, so the KMS orthogonality of the $X_i$'s translates to
\[
\op{Tr}(\rho^{-\frac{1}{4}} V_i^{\ast} \rho^{\frac{1}{4}} \rho^{-\frac{1}{2}} \rho^{\frac{1}{4}} V_j \rho^{-\frac{1}{4}} \rho^{-\frac{1}{2}}) = \delta^2 \delta_{ij}. 
\]
Now note that the left-hand side is equal to $\op{Tr}(\rho^{-1} V_{i}^{\ast} V_{j})$ to conclude.
\end{proof}

We fix the quantum graph $\mathcal{G}=(M_n, A,\psi)$ and we will denote its quantum edge indicator as $\varepsilon$. It is equal to
\[
\varepsilon = \frac{1}{\delta^2} \sum_{ij,r} e_{ij}\rho^{-1} \otimes V_{r} e_{ji} V_{r}^{\ast}.
\]
The $M_n$-valued inner product becomes $\langle a\otimes b, c\otimes d\rangle:= \psi(a^{\ast}c) b^{\ast}d$ and we also get $\langle x\cdot \varepsilon, y\cdot \varepsilon \rangle = \frac{1}{\delta^2} A(x^{\ast}y)$ (see \cite[Theorem 2.5 ]{BHINW}). By a direct computation we check that
\[
\varepsilon = \frac{1}{\delta^2} \sum_{i,j,r} V_{r}^{\ast} e_{ij} \rho^{-1} \otimes V_{r} e_{ji}.
\]
Note that elements of the form $e_{st}\varepsilon = \frac{1}{\delta^2} \sum_{r,j} e_{sj}\rho^{-1} \otimes V_{r} e_{jt} V_{r}^{\ast}$ span the edge correspondence as a right $M_n$-module. For $1\leq a\leq n$ and appropriate $r$ we define
\[
\psi_{ar}:= \sum_{k} e_{ak} V_{r} \rho^{-1} \varepsilon e_{k1} = \frac{1}{\delta^2}\sum_{k,i,j,q} e_{ak} V_{r} \rho^{-1} V_{q}^{\ast} e_{ij} \rho^{-1} \otimes V_{q} e_{ji}e_{k1} \in E_\G. 
\]
Because $e_{ji} e_{k1} = \delta_{ik} e_{j1}$, the left leg of the tensor product becomes $\op{Tr}(V_r \rho^{-1} V_{q}^{\ast}) e_{aj} \rho^{-1} = \delta^2 \delta_{rq} e_{aj} \rho^{-1}$. It follows that 
\[
\psi_{ar} = \sum_{j} e_{aj}\rho^{-1} \otimes V_{r} e_{j1}.
\]
Let us check that these elements span the edge correspondence as a right $M_n$-module. We have
\[
\sum_{r} \psi_{ar} e_{1t} V_{r}^{\ast} = \delta^2 e_{at} \varepsilon,
\]
and we know that these elements span the edge correspondence. We will now check that the set $(\psi_{ar})_{a,r}$ is orthogonal:
\begin{align*}
\langle \psi_{a_1 r_1}, \psi_{a_2 r_2}\rangle &= \sum_{j_1, j_2} \psi(\rho^{-1} e_{j_1 a_1} e_{a_2 j_2} \rho^{-1})e_{1j_1} V_{r_1}^{\ast} V_{r_2} e_{j_2 1} \\
&= \sum_{j_1, j_2} \delta_{a_1 a_2} \delta_{j_1 j_2} (\rho^{-1})_{j_1} e_{1j_1} V_{r_1}^{\ast} V_{r_2} e_{j_2 1}, 
\end{align*}
where we used the fact that $\rho$ is a diagonal matrix. Using it once again, we note that $e_{1j} (\rho^{-1})_{j} = e_{1j} \rho^{-1}$, hence
\[
\langle \psi_{a_1 r_1}, \psi_{a_2 r_2}\rangle= \delta_{a_1 a_2} e_{11} \op{Tr}(\rho^{-1} V_{r_1}^{\ast} V_{r_2}) = \delta^2 \delta_{a_1 a_2} \delta_{r_1 r_2} e_{11}.
\]
 If we consider $\phi_{ar}:= \frac{1}{\delta} \psi_{ar}$, then we get an orthogonal set such that $\langle \phi_{ar}, \phi_{ar}\rangle = e_{11}$ is a projection, so we get a (quasi-)orthonormal basis of the edge correspondence. It follows from \cite[Theorem 1.1.]{LN} that 
 \[\op{Tr}_{\tau} (x) = \sum_{a,r} \tau(\langle \phi_{ar}, x \phi_{ar}\rangle),\] where $\tau$ is the normalized trace on $M_n$. We can compute this expression explicitly:
\begin{align*}
\sum_{a,r} \langle \phi_{ar}, x\phi_{ar}\rangle &= \frac{1}{\delta^2}\sum_{j_1, j_2, a, r} \op{Tr}(\rho \rho^{-1} e_{j_1 a} x e_{a j_2} \rho^{-1}) e_{1 j_1} V_{r}^{\ast} V_{r} e_{j_2 1} \\
&= \op{Tr}(x) \frac{1}{\delta^2} \sum_{j,r} e_{11} \op{Tr}(\rho^{-1} V_{r}^{\ast} V_{r}) = \op{Tr}(x) e_{11} \op{dim}\op{span}\{V_r\}.
\end{align*} 
After applying the normalized trace we obtain 
\[
\op{Tr}_{\tau}(x) = \tau(x) \op{dim}\op{span}\{V_r\}.
\]
Note that we have $\langle \rho^{-1}, A(\rho^{-1})\rangle = \sum_{r} \op{Tr}(\rho \rho^{-1} V_{r} \rho^{-1} V_{r}^{\ast}) = \sum_{r} \op{Tr}(\rho^{-1} V_{r}^{\ast} V_{r}) = \delta^2 \op{dim}\op{span}\{V_r\}$, so we can easily express this dimension using the quantum adjacency matrix and we see how it is related to the matrix $D$.

We will now consider the case where $\B$ is a multi-matrix algebra. We have $\B:= \bigoplus_{a=1}^{d} M_{n_{a}}$, so the edge correspondence is a sub-bimodule of $\B\otimes \B = \bigoplus_{a,b=1}^{d} M_{n_{a}} \otimes M_{n_{b}}$. Because of that, the edge correspondence $E_\G$ naturally splits into a direct sum $E_{ab}$ of $M_{n_{a}}$-$M_{n_b}$-bimodules that are mutually orthogonal. We will use that to construct an orthonormal basis of $E_\G$. Let us denote by $A_{ba}$ the part of the quantum adjacency matrix mapping $M_{n_{a}}$ to $M_{n_b}$, i.e. we restrict $A$ to $M_{n_a}$ and then project onto $M_{n_b}$. We have a Kraus decomposition $A_{ba}(x) := \sum_{r} V_{ba}^{r} x (V_{ba}^{r})^{\ast}$, where the Kraus operators $V_{ba}^r$ are $n_b\times n_a$ matrices and satisfy 
\[\op{Tr}(\rho_{a}^{-1} (V_{ba}^{r})^{\ast} V_{ba}^{q}) = \op{Tr}(V_{ba}^{q} \rho_{a}^{-1}(V_{ba}^{r})^{\ast} )=\delta^2 \delta_{rq}.\]  The edge indicator is given by
\begin{align*}
    \varepsilon_\G = \frac{1}{\delta^2} \sum_{a,b,i,j,r} e_{ij}^a\rho_a^{-1} \otimes V_{ba}^r e_{ji}^a (V_{ba}^r)^{\ast}
\end{align*}
which can further be written as follows:
\begin{align*}
     \varepsilon_\G = \frac{1}{\delta^2} \sum_{a,b,i,j,r} (V_{ba}^r)^*e_{ij}^{ba}\rho_a^{-1} \otimes V_{ba}^r e_{ji}^{ab}.
\end{align*}
where $e_{k\ell}^{ab}$ denotes the matrix unit of $n_a\times n_b$ matrices with $1\leq k\leq n_{a}$ and $1\leq \ell \leq n_{b}$.
One can obtain the above expression by writing $V_{ba}^r=\sum_{k,\ell} (V_{ba}^r)_{k\ell}e_{k\ell}^{ba}$ and noting that $e_{ji}^{ab}e_{k\ell}^{bc}=\delta_{ik}e^{ac}_{j\ell}$.
We define elements
\[
\phi_{ab}^{ir}:= \frac{1}{\delta} \sum_{k} e_{ik}^{a} \rho_{a}^{-1} \otimes V_{ba}^{r} e_{k1}^{ab},
\]
 so that $\phi_{ab}^{ir} \in M_{n_a} \otimes M_{n_b}$. As above, one can show that $\phi_{ab}^{ir}\in E_\G$. Indeed, a direct calculation using $\Tr(V_{ba}^q\rho_a^{-1}(V_{ba}^r)^*)=\delta^2\delta_{rq}$ and the  expressions above for $\varepsilon_\G$ shows that
\begin{align*}
   \phi_{ab}^{ir}=  \frac{1}{\delta}\sum_{k} e_{
   ik}^{ab} V_{ba}^r\rho_a^{-1}\varepsilon_\G e_{k1}^b\in E_\G.
\end{align*}
It is clear that if the indices $a$ or $b$ are different then we get orthogonal elements, because the products are zero, so it suffices to see what happens in a given block. We compute
\begin{align*}
\langle \phi_{ab}^{i_1 r_1}, \phi_{ab}^{i_2 r_2}\rangle &= \frac{1}{\delta^2} \sum_{k_1, k_2} \op{Tr}(\rho_{a} \rho_{a}^{-1} e_{k_1 i_1} e_{i_2 k_2} \rho_{a}^{-1}) e_{1k_{1}}^{ba} (V_{ba}^{r_1})^{\ast} V_{ba}^{r_2} e_{k_{2}1} \\
&= \delta_{i_1 i_2} \delta_{r_1 r_2} e_{11}^{b}, 
\end{align*}
where the computation is exactly the same as in the single block case. This leads us to the next proposition.

\begin{proposition}
Let $\G=(\B, A, \psi)$ be a quantum graph and let $E_{\G}$ be its edge correspondence. Suppose that $\B\simeq \bigoplus_{a=1}^{d} M_{n_{a}}$ and for each pair $(a,b) \in \{1,\ldots,d\}^2$ let $(V_{ba}^{r})_{r}$ be Kraus operators of $A_{ba}: M_{n_{a}} \to M_{n_b}$. Define $\phi_{ab}^{ir}:= \frac{1}{\delta}\sum_{k} e_{ik}^{a} \rho_{a}^{-1}\otimes V_{ba}^{r}e_{k1}^{ab}$. Then the family $(\phi_{ab}^{ir})_{a,b,i,r}$ is an orthonormal basis of the edge correspondence $E_{\G}$. 
\end{proposition}

We can now use this orthonormal basis to find KMS states. Any tracial state on $\B$ is given by $\tau_{\bm{\lambda}}:= \bigoplus_{a=1}^{d} \lambda_{a} \Tr_{n_a}$, where $\bm{\lambda}:= (\lambda_1, \dots, \lambda_n)$ satisfies $\sum_{a=1}^d \lambda_a n_a = 1$. For any $x=\oplus_{a=1}^d x_a\in\B$  we have 
\begin{align*}
\op{Tr}_{\tau_{\bm{\lambda}}}(x) &= \sum_{a,b,i,r}\tau_{\bm{\lambda}} (\langle \phi_{ab}^{ir}, x\phi_{ab}^{ir}\rangle) \\
&= \frac{1}{\delta^2} \sum_{k_1, k_2, a, b, i, r} \op{Tr}(\rho_{a} \rho_{a}^{-1} e_{k_1 i}^{a} x e_{i k_2} \rho_{a}^{-1})\tau_{\bm{\lambda}} (e_{1 k_1}^{ba} (V_{ba}^{r})^{\ast} V_{ba}^{r} e_{k_2 1}^{ab}) \\
&= \sum_{a,b} \op{Tr}_{n_a}(x_{a}) \op{dim}\op{span}\{V_{ba}^{r}: r\} \lambda_b
\end{align*}
The KMS condition becomes
\[
\sum_{a,b} \op{Tr}_{n_a}(x_{a}) \op{dim}\op{span}\{V_{ba}^{r}:r\} \lambda_b = e^{\beta} \sum_{a} \lambda_{a} \op{Tr}_{n_a}(x_{a})
\]
Given the fact that the numbers $\op{Tr}_{n_{a}} (x_{a})$ are arbitrary, we conclude that the vector $\bm{\lambda}$ is a (non-negative) eigenvector of the matrix $[T_{ab}]_{1\leq a,b\leq d}$ where $T_{ab}:= \op{dim}\op{span}\{V_{ba}^{r}\}$ with the eigenvalue $e^{\beta}$. We also note that $T_{ab} = \frac{1}{\delta^2} \langle \rho_{b}^{-1}, A\rho_{a}^{-1}\rangle = \frac{1}{\delta^2} D_{ab}$, which shows why the entries of $\frac{1}{\delta^2} D$ are integers. This reproves Theorem \ref{Thm:KMS} and provides an interpretation for the calculations carried out in Subsection \ref{subsec:int}.

\begin{example}
We can also handle KMS states for more general actions. Since the edge correspondence $E_{\G}$ naturally splits into a direct sum $\oplus_{a,b=1}^d E_{ab}$ of orthogonal $M_{n_a}$-$M_{n_b}$ bimodules, for each array of numbers $(N_{ab})_{a,b}$ with $N_{ab}>1$ we can consider the action $\Phi_t$ on $E_{\G}$, where $\Phi_t$ acts by multiplication by $N_{ab}^{it}$ on the $(a,b)$ component $E_{ab}$. The computations are very similar to the standard gauge action and the condition we find is that
\[
\sum_{b} T_{ab} N_{ab}^{-\beta}\lambda_{b} = \lambda_{a},
\]
i.e. the vector $(\lambda_{1},\dots, \lambda_d)$ is an eigenvector of the matrix $(T_{ab} N_{ab}^{-\beta})_{a,b}$ with eigenvalue $1$.

\end{example}

\section{KMS states (non-tracial cases)}

In the non-tracial case one can use \cite[Theorem 3.2 and Theorem 3.5]{LN} to find a condition for KMS states. Once again there is a procedure to induce a weight $\kappa_{\varphi}$ on $\B(X)$ from a state $\varphi$ on $\B$ and the KMS condition is that $\kappa_{\varphi}$ restricted to $\B$ (where $\B$ acts on $X$ from the left) is equal to $\varphi$. Because we have a formula for the left action of $\B$ as compact operators on $X$, one can explicitly write down a condition for KMS states and in some simple cases we will use it.

Once again, let $\G=(\B,A,\psi)$ be a quantum graph and let $E_{\G}$ be its edge correspondence. We define the isometry group on $E_{\G}$ as $U_t := e^{it} (\sigma_{-t} \otimes \sigma_{-t})$, where $\sigma_{t}$ is the modular group of $\psi$. Note that  we added the minus sign so that $\psi$ is a KMS state with $\beta=1$, not $\beta=-1$. Since we want to assume that $E_{\G}$ is preserved by $U_{t}$, we will assume that $A$ commutes with the modular group. Suppose now that $\varphi$ is a $\beta$-KMS state for the action $\sigma_{-t}$ on $\B$. Consider first the case $\B \simeq M_n$. Then $\psi(x) =\op{Tr}(\rho x)$ and $\varphi(x) = \op{Tr}(\sigma x)$. The $(\sigma_{-t},\beta)$-KMS condition for $\varphi$ means that for all $a,b\in M_n$ we have
\[
\op{Tr}(\sigma a b) = \op{Tr}(\sigma b \rho^{\beta} a \rho^{-\beta}).
\]
By traciality this is equivalent to
\[
\op{Tr}(\sigma a b) = \op{Tr}(\rho^{\beta} a \rho^{-\beta} \sigma b),
 \]
 so $\sigma a = \rho^{\beta} a \rho^{-\beta}\sigma$, as it happens for all matrices $b$. It follows that $\sigma$ is a scalar multiple of $\rho^{\beta}$; the scalar is unique as $\sigma$ is positive and of trace one. It follows that in the case of a multimatrix algebra on each block we have only one choice for the state and we can just vary the weights, just like in the tracial case. So any $\beta$-KMS $\varphi$ on $\B$ is of the form $\bigoplus_{a=1}^{d} \lambda_{a} \op{Tr}(\rho_{a}^{\beta} \cdot)$, where $\sum_{a=1}^{d} \lambda_a \op{Tr}(\rho_a^{\beta}) = 1$. Fix such a $\varphi$ and look at the induced functional $\kappa_{\varphi}$ on $B(E_\G)$, which on rank one operators is given by $\kappa_{\varphi}(|\xi\rangle\langle \eta|) = \varphi(\langle U_{\frac{i\beta}{2}}\eta, U_{\frac{i \beta}{2}} \xi\rangle)$. By \cite[Theorem  2.5 and Theorem 2.12]{BHINW} we obtain
 \begin{align*}
\kappa_{\varphi}(f_{ij}^{a}) &= \sum_{k=1}^{n_{a}} \varphi(\langle U_{\frac{i\beta}{2}} f_{jk}^{a}\varepsilon_{\G}, U_{\frac{i\beta}{2}}f_{ik}\varepsilon_{\G}\rangle) \\
&= \sum_{k} e^{-\beta} \varphi(\langle \sigma_{-\frac{i\beta}{2}}(f_{jk}^{a}) \varepsilon_{\G}, \sigma_{-\frac{i \beta}{2}}(f_{ik}^{a}) \varepsilon_{\G}\rangle) \\
&= \frac{1}{\delta^2}e^{-\beta} \sum_{k} \varphi(A( (\sigma_{-\frac{i\beta}{2}}(f_{jk}^{a}))^{\ast} \sigma_{-\frac{i \beta}{2}}(f_{ik}^{a}))).
 \end{align*}
We assume that the density matrix of $\psi$ is diagonal, hence we can compute 
\[\sigma_{-\frac{i\beta}{2}}(f_{jk}^{a}) = \rho^{\frac{\beta}{2}}f_{jk}^{a}\rho^{-\frac{\beta}{2}}=(\psi(e_{jj}^{a}))^{\frac{\beta}{2}} (\psi(e_{kk}^{a}))^{-\frac{\beta}{2}}f_{jk}^{a},\] 
so we arrive at the expression
\[
\frac{e^{-\beta}}{\delta^2} \sum_{k} (\psi(e_{jj}^{a}) \psi(e_{ii}^{a}))^{\frac{\beta}{2}} (\psi(e_{kk}^{a}))^{-\beta} \frac{1}{\psi(e_{ii}^{a})} \varphi(A(\delta_{ij} f_{kk}^{a})).
\]
This, in turn, is equal to 
\[
\frac{e^{-\beta}}{\delta^2} \delta_{ij} (\psi(e_{ii}^{a}))^{\beta-1} \sum_{k} (\psi(e_{kk}^{a}))^{-\beta} \varphi(A(f_{kk}^{a})).
\]
Now note that $\sum_{k} \frac{f_{kk}^{a}}{\psi(e_{kk}^{a})^{\beta}} = \rho_{a}^{-\beta-1}$, so we finally obtain
\[
\kappa_{\varphi}(f_{ij}^{a}) = \delta_{ij} \frac{e^{-\beta}}{\delta^2} (\psi(e_{ii}^{a})^{\beta-1} \varphi(A(\rho_{a}^{-\beta-1})).
\]
On the other hand $\varphi(f_{ij}^{a}) = \delta_{ij}\lambda_{a} (\psi(e_{ii}^{a}))^{\beta-1}$. By noting that $\varphi(X)=\sum_{b=1}^d\lambda_b\langle\rho_b^{\beta-1},X\rangle_\psi$ for all $X\in \B$, the KMS condition gives us the equation
\[
\frac{e^{-\beta}}{\delta^2} \sum_{b} \lambda_{b} \langle \rho_b^{\beta-1}, A (\rho_a^{-\beta-1})\rangle_{\psi} = \lambda_{a}, 
\]
which means that $(\lambda_1,\dots,\lambda_d)$ is an eigenvector of the matrix $[D_{ab}]$ where $D_{ab}= \langle \rho_b^{\beta-1}, A (\rho_a^{-\beta-1})\rangle_{\psi}$. This matrix has non-negative entries, because they can be expressed as traces of products of two positive matrices. Note that, unlike the tracial case, $\beta$ not only appears in the eigenvalue but also in the matrix itself. 

Let us see what happens in the case of a complete quantum graph, i.e. $Ax = \delta^2 \psi(x) \mathds{1}$. Then the left-hand side is equal to
\[
e^{-\beta} \sum_{b} \lambda_b \op{Tr}(\rho_b^{\beta}) \op{Tr}(\rho_a^{-\beta}) = e^{-\beta} \op{Tr} (\rho_a^{-\beta}),
\]
because $\sum_{a=1}^{d} \lambda_a \op{Tr}(\rho_a^{\beta}) = 1$. It follows that $\lambda_a = e^{-\beta} \op{Tr}(\rho_{a}^{-\beta})$. But from the normalization condition for $\lambda_a$'s we obtain the equation for $\beta$:
\[
\sum_{a} \op{Tr}(\rho_a^{\beta})\op{Tr}(\rho_a^{-\beta}) = e^{\beta},
\]
which is a nonlinear equation and might in general not have a solution, even in the case $d=1$, i.e. a single matrix block. Indeed, consider $M_2$ equipped with a diagonal density matrix (with entries $t$ and $1-t$). Then the equation becomes
\[
2 + (\frac{t}{1-t})^{\beta} + (\frac{1-t}{t})^{\beta} = e^{\beta}.
\]
For $\beta=0$ the left-hand side is bigger, so if any of the ratios $\frac{1-t}{t}$ or $\frac{t}{1-t}$ are at least equal to $e$ then the left-hand side will be larger for any $\beta>0$, so there will not be a solution. On the other hand, if both ratios are strictly smaller than $e$, then eventually the right-hand side will become larger, so there will be some solution by continuity.

\section*{Acknowledgments}
The first-named author was partially supported by the Research Foundation - Flanders (FWO) through a Postdoctoral fellowship (1221025N).
The second-named author was partially supported by the National Science Center, Poland (NCN) grant no. 2021/43/D/ST1/01446.

The project is co-financed by the Polish National Agency for Academic Exchange within the Polish Returns Programme. 

\vspace{5 pt}
\includegraphics[scale=0.5]{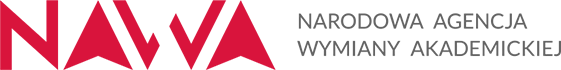}

\end{document}